\documentclass[a4paper,12pt]{amsart}
\usepackage{amssymb}
\usepackage{cite}
\usepackage{ifthen}
\usepackage[dvips]{graphicx}
\nonstopmode \numberwithin{equation}{section}
\newtheorem*{theoA}{Theorem A}
\newtheorem*{theoB}{Theorem B}

\newtheorem*{theoD}{Theorem D}

\usepackage{amssymb}
\usepackage{ifthen}
\usepackage{graphicx}
\usepackage{amsmath}
\usepackage[T1]{fontenc} %skandit
\usepackage[utf8]{inputenc}
\usepackage[usenames,dvipsnames]{color}
\usepackage{color}
\usepackage[english]{babel}
\usepackage{fancyhdr}
\usepackage{fancybox}
\usepackage{tikz}
\theoremstyle{plain}
\newtheorem{prop}{Proposition}

\newtheorem{ques}{Question}

\newtheorem{conj}{Conjecture}

\theoremstyle{definition}
\newtheorem{defi}{Definition}[section]

\newtheorem{cor}{Corollary}[section]
\newtheorem{thm}{Theorem}[section]

\newtheorem{lem}{Lemma}[section]
\newtheorem{prob}{Problem}
\newtheorem{rem}{Remark}[section]

\theoremstyle{plain}

\newtheorem*{thmC}{Theorem C}

\newtheorem*{lemA}{Lemma A}
\newtheorem*{lemB}{Lemma B}

\newcounter{minutes}
\divide\time by 60
\newcounter{hours}
\multiply\time by 60
\addtocounter{minutes}{-\time}
\newcounter {own}
\def\theown {\thesection       .\arabic{own}}

\newenvironment{pf}[1][]{%
	\vskip 3mm
	\noindent
	\ifthenelse{\equal{#1}{}}%
	{{\slshape Proof. }}%
	{{\slshape #1.} }%
}%
{\qed\bigskip}

\newcounter{alphabet}

\newcommand{\A}{{\mathcal A}}

\def\be{\begin{equation}}
	\def\ee{\end{equation}}

\newcommand{\bee}{\begin{enumerate}}
	\newcommand{\eee}{\end{enumerate}}

\newcommand{\blem}{\begin{lem}}
	\newcommand{\elem}{\end{lem}}
\newcommand{\bthm}{\begin{thm}}
	\newcommand{\ethm}{\end{thm}}
\newcommand{\bcor}{\begin{cor}}
	\newcommand{\ecor}{\end{cor}}
\newcommand{\beg}{\begin{examp}}
	\newcommand{\eeg}{\end{examp}}
\newcommand{\begs}{\begin{examples}}
	\newcommand{\eegs}{\end{examples}}

\newcommand{\bdefn}{\begin{defn}}
	\newcommand{\edefn}{\end{defn}}

\newcommand{\bprob}{\begin{prob}}
	\newcommand{\eprob}{\end{prob}}
\newcommand{\bei}{\begin{itemize}}
	\newcommand{\eei}{\end{itemize}}

\newcommand{\bcon}{\begin{conj}}
	\newcommand{\econ}{\end{conj}}
\newcommand{\bcons}{\begin{conjs}}
	\newcommand{\econs}{\end{conjs}}
\newcommand{\bprop}{\begin{prop}}
	\newcommand{\eprop}{\end{prop}}
\newcommand{\br}{\begin{rem}}
	\newcommand{\er}{\end{rem}}
\newcommand{\brs}{\begin{rems}}
	\newcommand{\ers}{\end{rems}}
\newcommand{\bo}{\begin{obser}}
	\newcommand{\eo}{\end{obser}}
\newcommand{\bos}{\begin{obsers}}
	\newcommand{\eos}{\end{obsers}}
\newcommand{\bpf}{\begin{pf}}
	\newcommand{\epf}{\end{pf}}
\newcommand{\ba}{\begin{array}}
	\newcommand{\ea}{\end{array}}
\newcommand{\beq}{\begin{eqnarray}}
	\newcommand{\beqq}{\begin{eqnarray*}}
		\newcommand{\eeq}{\end{eqnarray}}
	\newcommand{\eeqq}{\end{eqnarray*}}

\begin{document}

\title{Bohr radius for invariant families of bounded analytic functions and certain Integral transforms}

\author{Molla Basir Ahamed}
\address{Molla Basir Ahamed, Department of Mathematics, Jadavpur University, Kolkata-700032, West Bengal, India.}
\email{mbahamed.math@jadavpuruniversity.in}

\author{Partha Pratim Roy}
\address{Partha Pratim Roy, Department of Mathematics, Jadavpur University, Kolkata-700032, West Bengal, India.}
\email{parthacob2023@gmail.com}

\author{Sabir Ahammed}
\address{Sabir Ahammed, Department of Mathematics, Jadavpur University, Kolkata-700032, West Bengal, India.}
\email{sabira.math.rs@jadavpuruniversity.in}

\subjclass[{AMS} Subject Classification:]{Primary 30C45, 30C50,30C65, 30C80, 44A55}

\def\thefootnote{}
\footnotetext{ {\tiny File:~\jobname.tex,
printed: \number\year-\number\month-\number\day,
          \thehours.\ifnum\theminutes<10{0}\fi\theminutes }
}\makeatletter\def\thefootnote{\@arabic\c@footnote}\makeatother
\begin{abstract} 	
 In this paper, we first obtain a refined Bohr radius for invariant families of bounded analytic functions on the unit disk $ \mathbb{D} $. Then, we obtain Bohr inequality for certain integral transforms, namely Fourier (discrete) and Laplace (discrete) transforms of bounded analytic functions $ f(z)=\sum_{n=0}^{\infty}a_nz^n $, in a simply connected domain \begin{align*}
 		\Omega_\gamma:=\biggl\{z\in\mathbb{C}: \bigg|z+\dfrac{\gamma}{1-\gamma}\bigg|<\dfrac{1}{1-\gamma}\;\mbox{for}\; 0\leq \gamma<1\biggr\},
 \end{align*}
where $ \Omega_0=\mathbb{D} $. These results generalize some existing results. We also show that a better estimate can be obtained in radius and inequality can be shown sharp for Laplace transform of $ f $.
\end{abstract}
\maketitle
\pagestyle{myheadings}
\markboth{M. B. Ahamed, P. P. Roy and S. Ahammed}{Bohr radius for invariant families of bounded analytic functions and certain Integral transforms}

\section{introduction}
Let $\mathcal{B}$ be the set of all analytic functions of the form $ f(z)=\sum_{k=0}^{\infty}a_kz^k $ in the unit disk $\mathbb{D}:=\{z\in\mathbb{C} : |z|<1\}$ such that $|f(z)|\leq 1$ for all $z\in\mathbb{D}$. For $ f\in\mathcal{B} $, we define $ ||f||_{\infty}:=\sup_{z\in\mathbb{D}}|f(z)| $. We recall a classical theorem of  Harald Bohr for the class $\mathcal{B}$.
\begin{theoA}(see \cite{Bohr-1914})
	Let $f\in \mathcal{B}$ with the power series expansion  $f(z)=\sum_{n=0}^{\infty}a_nz^n$. If $||f||_{\infty}\leq 1$, then
	\begin{align}\label{Eq-1.1}
		M_f(r):=\sum_{n=0}^{\infty}|a_n|r^n\leq ||f||_{\infty}\; \mbox{for}\; |z|=r\leq\frac{1}{3}.
	\end{align}
	The constant $1/3$ is the best possible.
\end{theoA}
 The constant $1/3$ is known as the classical Bohr radius, and \eqref{Eq-1.1} is known as the Bohr inequality for the class $\mathcal{B}$. In recent years, a significant amount of research has been devoted to improving, refining, or generalizing the classical Bohr inequality for different functional settings. Finding the Bohr radius for certain classes of functions has become an active area of research, and this study is known in the literature as the Bohr phenomenon. However, it is important to note that not every class of functions exhibits the Bohr phenomenon. For multidimensional study of the Bohr radius for holomorphic mappings, we refer to the articles \cite{Aizen-PAMS-1999,Boas-Khavinson-1997,Hamada-IJM-2009,Galicer-TAMS-2021,Kumar-PAMS-2022,Kumar-Manna-JMAA-2023} and references therein.\vspace{1.2mm}
 
 The classical theorem of Bohr actually gained significance after Dixon's work (see \cite{Dixon & BLMS & 1995}), where it was utilized to disprove a conjecture regarding the non-unital von Neumann inequality for Banach algebras. The exploration of the Bohr inequality for different classes of  functions in one as well several complex variables, and functional contexts has emerged as a highly engaging field of study in modern function theory, prompting extensive research efforts by numerous scholars in recent years. For different aspects of the Bohr inequality including recent progress in the topic, the reader is referred to the following articles  \cite{Alkh-Kay-Pon-PAMS-2019,Paulsen-PLMS-2002,Paulsen-PAMS-2004,Paulsen-BLMS-2006,Ismagilov-2020-JMAA,Kayumov-CRACAD-2018,Kayumov-Khammatova-JMAA-2021,Kayumov-MJM-2022,Kay & Pon & AASFM & 2019,Lata-Singh-PAMS-2022,Liu-JMAA-2021,Liu-Ponnusamy-MN,Liu-Ponnusamy-PAMS-2021,Ponnusamy-HJM-2021} and references therein. \vspace{1.2mm}

Similar to the Bohr radius and Bohr inequality, there is also a concept of Rogosinski radius and Rogosinski inequality in the literature. A combined inequality is known as the Bohr-Rogosinski inequality corresponding to a class of functions. Kayumov \emph{et al.} (see \cite{Kayumov-Khammatova-JMAA-2021}) established the Bohr-Rogosinski inequality in a generalized context and derived the Bohr-Rogosinski radii for Ces$\acute{a}$ro operators for bounded analytic functions on unit disk $\mathbb{D}$. In 2021, Liu \emph{et al.} (see \cite{Liu-Liu-Ponnusamy-BSM-2021}) introduced a refined version of the Bohr and Bohr-Rogosinski inequality in a broader perspective. Similarly, Kumar and Sahoo (see \cite{Kumar-Sahoo-MJM-2021}) determined the sharp Bohr-type radii for certain complex integral operators defined on a set of bounded analytic functions in  $\mathbb{D}$. Later, Allu and Ghosh (see \cite{Allu-Ghosh-2023}) investigated the Bohr-type inequality for Ces$\acute{a}$ro operators and Bernardi integral operators acting on analytic functions defined within a simply connected domain that includes the unit disk $\mathbb{D}$ in 2023. Meanwhile, Kayumov \textit{et al.} \cite{Kayumov-Khammatova-JMAA-2021} obtained the Bohr inequality by examining the integral representation of Ces$\acute{a}$ro operators on the set of holomorphic functions defined on the unit disk $\mathbb{D}$ and also analyzed the asymptotic behavior of the corresponding Bohr sum $ \sum_{n=0}^{\infty}|a_n|r^n $.\vspace{1.2mm}														

For $0\leq\gamma<1$, we consider the disk $\Omega_\gamma$ defined by 
\begin{align*}
	\Omega_\gamma=\biggl\{z\in\mathbb{C}: \bigg|z+\dfrac{\gamma}{1-\gamma}\bigg|<\dfrac{1}{1-\gamma}\biggr\}.
\end{align*}
\begin{figure}[!htb]
	\begin{center}
	\includegraphics[width=0.55\linewidth]{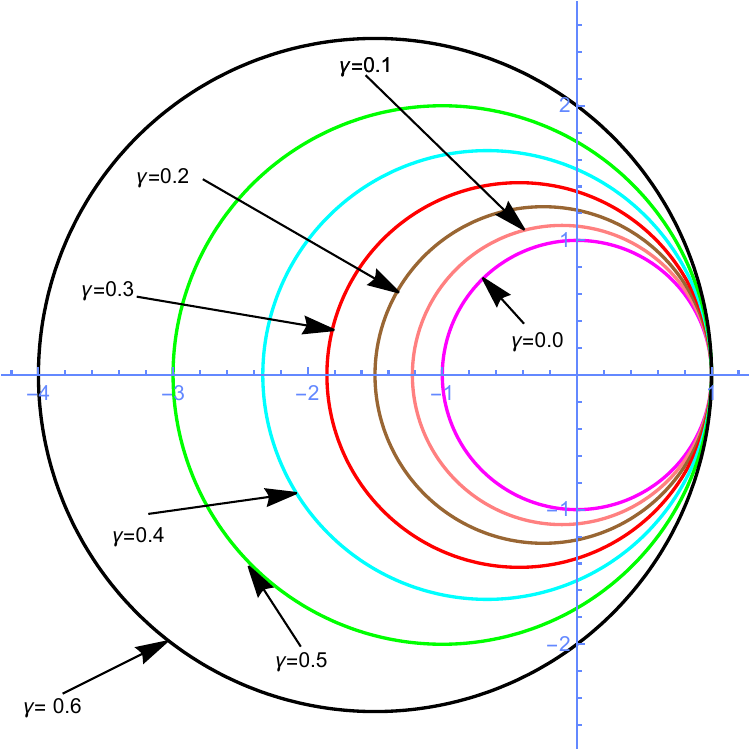}
	\end{center}
	\caption{The graph of the circle $C_{\gamma}:=\partial\Omega_{\gamma}$ for different values of $\gamma=0.0, 0.1, 0.2, 0.3, 0.4, 0.5$ and $0.6$.}
\end{figure}
It is  easy to see that the unit disk $\Omega_0=\mathbb{D}$ and for $ 0<\gamma<1 $, $ \mathbb{D}\subset \Omega_{\gamma} $. Several studies on the Bohr radius problem have been conducted for functions on $\Omega_\gamma$ and other similar classes. In $ 2010 $, Fournier and Ruscheweyh (see \cite{Four-Rusc-2010}) introduced the Bohr inequality specifically for the simply connected domain $\Omega_\gamma$ and obtained the following result as a generalization of the classical Bohr inequality.
\begin{theoB}\cite[Theorem 1]{Four-Rusc-2010} For $0\leq \gamma< 1$, let $f\in\mathcal{B}(\Omega_{\gamma})$ with the power series expansion  $f(z)=\sum_{n=0}^{\infty}a_nz^n$ in $\mathbb{D}$. Then,
\begin{align}\label{EQ-1.2}
	\sum_{n=0}^{\infty}|a_n|r^n\leq1\;\;\mbox{for}\;\;r\leq\rho_\gamma:=\frac{1+\gamma}{3+\gamma}.
\end{align}
Moreover, $\sum_{n=0}^{\infty}|a_n|\rho_\gamma^n=1$ holds for a function $f(z)=\sum_{n=0}^{\infty}a_nz^n$ in $\mathcal{B}(\Omega_{\gamma})$ if and only if $f(z)=c$ with $|c|=1$.
\end{theoB} 
Utilizing the idea of Theorem B and its proof, an extensive research work are done by several authors for different classes of functions. For instance, Evdoridis \emph{et al.} (see \cite{Evd-Ponn-Rasi-2020}) improved the inequality \eqref{EQ-1.2}, extending it for class of harmonic functions in $\Omega_\gamma$. Subsequently, Ahamed \emph{et al.} (see \cite{Ahamed-AFM-2021}) improved the Bohr radius of those obtained in \cite{{Four-Rusc-2010,Evd-Ponn-Rasi-2020}} further, and find the Bohr-Rogosinski radius, and refined Bohr radius for the class of analytic functions in $\Omega_\gamma$. Recently, Kumar (see \cite{Kumar-CVEE-Integral-2023}) have generalized the Bohr inequality for the simply connected domain $\Omega\gamma$ by introducing a sequence $\{\phi_n(r)\}_{n=0}^{\infty}$ of continuous functions in $ (0, 1) $ and established results.\vspace{2mm}

In this paper, our aim is two fold: Firstly, we aim to obtain refined Bohr inequalities for an invariant class of bounded analytic functions. Secondly, we aim to determine the Bohr radius for the discrete Fourier transform on a simply connected domain $\Omega_{\gamma}$. The organization of the paper is as follows: In Section 2, we study refined Bohr inequalities by redefining the Bohr inequality as discussed in \cite{Ponn-Shmi-Star-2024} for a linearly invariant family. We show that the Bohr radius can be obtained as a root of an equation related to a function in the family. In Section 3, we obtain the Bohr radius for the discrete Fourier transform on $\Omega_\gamma$ and demonstrate that the corresponding Bohr inequality is sharp. Furthermore, our results generalize some existing results. The proof of the main results in each section, along with the background, are discussed separately in each section.
\section{Linearly invariant family and refined Bohr inequality}
In $1964$, Pommerenke (see \cite{Pommerenke-1964}) introduced the concept of linearly invariant family. Let $\mathcal{A}$ denote the class of all analytic functions $g$ defined on $\mathbb{D}$ of the form 
\begin{align}\label{Eq-2.1}
	g(z)=z+\sum_{n=2}^{\infty}a_n(g)z^n,\;z\in\mathbb{D}.
\end{align}
Then, the set $\mathfrak{M}$ of locally univalent functions $g\in\mathcal{A}$ is called a linear invariant family, also known as LIF. A linear family is defined as a set where, for any functions $g\in\mathfrak{M}$ and any conformal automorphism $\phi$ of $\mathbb{D}$, the function $g_\phi$ is defined by
\begin{align*}
	g_\phi(z)=\frac{g(\phi(z))-g(\phi(0))}{g^{\prime}(\phi(0))\phi^{\prime}(0)}=z+\cdots
\end{align*} 
also belongs to $\mathfrak{M}$. The order of a linearly invariant family is the number $ {\rm ord}\;\mathfrak{M}=\sup\limits_{g\in \mathfrak{M}}|a_2(g)|. $ A universally linearly invariant family of order $\alpha$, denoted by $\mathcal{U}_\alpha$, is the union of all LIF $\mathfrak{M}$ for which ${\rm ord}\;\mathfrak{M}\leq\alpha$.
It has been proven that for any LIF $\mathfrak{M}$, its order ${\rm ord}\;\mathfrak{M}\geq1$ (see \cite{Pommerenke-1964}). Additionally, $\mathcal{U}_1=\mathcal{K}$ represents the class of all convex analytic functions $g$ in $\mathbb{D}$ of the form \eqref{Eq-2.1}, which map $\mathbb{D}$ onto convex regions. The class of univalent functions $\mathcal{S}$ is also an example of a linearly invariant family (LIF). According to Bieberbach's theorem (see \cite{Bieberbach-1916}), we know that the order of this family is ${\rm ord}\;\mathcal{S}=2$. \vspace{1.2mm}

Several well-known classes of analytic functions in $\mathbb{D}$ are considered to be linear invariant families. It seems that the properties of analytic and locally univalent functions are primarily influenced by the order of their families rather than by their geometric characteristics.\vspace{1mm}

Let $\mathfrak{M}$ be an LIF of finite order. We define the family $\mathfrak{LM}$ as follows:
\begin{align*}
	\mathfrak{LM}:=\{f(z)=\log g(z)=\sum_{n=1}^{\infty}a_n(f)z^n:g\in \mathfrak{M}\},
\end{align*}
where the branch of logarithm of $g^\prime(z)$ is chosen such that $\log g^\prime(0)=0$. \vspace{1.5mm}

 Let $B$ be the class of analytic functions of the form  $f(z)=\sum_{n=1}^{\infty}a_n(f)z^n$ in $\mathbb{D}$, uniformly bounded on compact sets from $\mathbb{D}$, \textit{i.e.} for every compact set $K$ from $\mathbb{D}$ there is a constant, namely $M(K)>0$ such that for any
 function $f\in B$ one has $|f(z)|\leq M(K)$ for $z\in K$.\vspace{2mm}
 
In \cite{Ponn-Shmi-Star-2024}, Ponnusamy \emph{et al.} examined the Bohr radius for the class of linearly invariant family (LIF), \emph{e.g.} $\mathfrak{L}\mathcal{K}=\{\log f^{\prime} : f\in\mathcal{K}\}$ (see \cite[Proposition 3]{Ponn-Shmi-Star-2024}) and proved the following lemma.
 \begin{lemB}\cite[Lemma 1]{Ponn-Shmi-Star-2024}
 	Let $f(z)=\sum_{n=1}^{\infty}a_n(f)z^n\in B$. Denote 
 	\begin{align*}
 		A_n=\sup\limits_{f\in B}|a_n(f)|\;\;\mbox{and}\;\;R_*=\sup\biggl\{r:\sum_{n=1}^{\infty}A_nr^n\leq1\biggr\}.
 	\end{align*}
 \end{lemB}
 Then $R_*\leq\rho(B)$. If $B$ contains a function $F(z)=\sum_{n=1}^{\infty}a_n(F)z^n$ for which all coefficients $a_n(F)\geq0, \sup_{z\in\mathbb{D}}|F(z)|>1$, then $r_F$  is the only root of the equation $F(r) = 1$ on $(0, 1),$ then $\rho(B)\leq r_F.$\vspace{1.2mm}
 
 Refining the Bohr inequality for certain class of functions and showing that the inequality sharp is interesting as well as difficult task. To continue the research on Bohr radius for LIF, a natural question thus arises as follows.
 \begin{ques}\label{Q-1}
 	Can we establish a refined version of Lemma B?
 \end{ques}
 To answer the Question \ref{Q-1}, we need some definitions and notations. The following definition will help to prove our main results. 
 \begin{defi}\label{def-1.1}
 	For an arbitrary class $\mathcal{M}$ of analytic functions of the form $f(z)=\sum_{n=0}^{\infty}a_nz^n$ in $\mathbb{D}$ with $|a_0|<1,$ and $ \lambda:[0,1]\to[0,\infty)$ be a real valued function, the refined Bohr radius of the class $\mathcal{M}$ is defined as
 	\begin{align*}
 		\rho_*^\lambda(\mathcal{M}):=\sup\biggl\{\rho^\lambda=\rho^\lambda(\mathcal{M}):A^\lambda_f(|z|)\leq1\;\mbox{for}\;f\in \mathcal{M}, z\in\mathbb{D}_\rho=\{z\in\mathbb{D}:|z|<\rho\}\biggr\},
 	\end{align*}
 	where
 	\begin{align*}
 		A^\lambda_f(|z|):=\sum_{n=0}^{\infty}|a_n||z|^n+\lambda(r)\sum_{n=1}^{\infty}|a_n|^2|z|^{2n}.
 	\end{align*}
 \end{defi}
Inspired by the article \cite{Ponn-Shmi-Star-2024}, we establish a refined Bohr inequality for linearly invariant family (LIF), (\emph{e.g.} for the families $\mathfrak{L}\mathcal{K}$ and $\mathcal{S}$) by introducing a function $\lambda:[0,1]\to[0,\infty)$. In view of Definition \ref{def-1.1}, we first obtain the following lemma, which will help us later to establish a refined Bohr inequality for the families $\mathfrak{L}\mathcal{K}$ and $\mathcal{S}$.
\begin{lem}\label{lem-1.1}
	Let $f(z)=\sum_{n=1}^{\infty}a_n(f)z^n\in B$ and  $\lambda:[0,1]\to[0,\infty)$ be a real valued function with $A_n=\sup\limits_{f\in B}|a_n(f)|$ and 
	\begin{align*}
		R_*^\lambda=\sup\bigg\{r:\sum_{n=1}^{\infty}A_nr^n+\lambda(r)\sum_{n=2}^{\infty}A_n^2r^{2n}\leq 1\bigg\}
	\end{align*}
	Then $R_*^\lambda\leq\rho^\lambda(B)$.  If $B$ contains a function $F(z)=\sum_{n=1}^{\infty}a_n(F)z^n$ for which all coefficients $a_n(F)\geq0$, $\sup_{z\in\mathbb{D}}|F(z)|>1$ and $r_F^\lambda$ is the only root of the equation $G^\lambda_F(r)=1$ on $(0,1)$, then $\rho^\lambda(B)\leq r^\lambda_F$.
\end{lem}
 \begin{proof}[\bf Proof of Lemma \ref{lem-1.1}]
 For each $f\in B$, we denote 
 \begin{align*}
 r^\lambda_f=\sup\bigg\{r:\sum_{n=1}^{\infty}|a_n(f)|r^n+\lambda(r)\sum_{n=2}^{\infty}|a_n(f)|^2r^{2n}\leq 1\bigg\}.
 \end{align*}
In view of Definition \ref{def-1.1}, we see that $\rho^\lambda(B)=\inf\limits_{f\in B}r^\lambda_f$. Further, definition of $A_n$ allows us to obtain the following inequalities
 \begin{align*}
 \sum_{n=1}^{\infty}|a_n(f)|r^n\leq\sum_{n=1}^{\infty}A_nr^n\; \mbox{and}\; \sum_{n=2}^{\infty}|a_n(f)|^2r^{2n}\leq\sum_{n=2}^{\infty}A_n^2r^{2n}\;\mbox{for all}\;f\in B.
 \end{align*}
 Thus, combining these two inequalities, we easily obtain
 	\begin{align*}
 		\sum_{n=1}^{\infty}|a_n(f)|r^n+\lambda(r)\sum_{n=2}^{\infty}|a_n(f)|^2r^{2n}\leq\sum_{n=1}^{\infty}A_nr^n+\lambda(r)\sum_{n=2}^{\infty}A_n^2r^{2n}.
 	\end{align*}
 	Consequently, for any $f\in B$, we see that
 	\begin{align*}
 		&\sup\bigg\{r:	\sum_{n=1}^{\infty}|a_n(f)|r^n+\lambda(r)\sum_{n=2}^{\infty}|a_n(f)|^2r^{2n}\bigg\}\geq\sup\bigg\{r:\sum_{n=1}^{\infty}A_nr^n+\lambda(r)\sum_{n=2}^{\infty}A_n^2r^{2n}\bigg\}\\&\iff r^\lambda_f\geq R^\lambda_*.
 	\end{align*}
 	Thus, it follows that $R_*^\lambda\leq\rho^\lambda(B)$. Since $\rho^\lambda(B)\leq\inf\limits_{f\in B}r^\lambda_f$, we have the following inequality for the function $F$ from the formulation of the lemma 
 \begin{align*}
 \rho^\lambda(B)\leq&\sup\bigg\{r:\sum_{n=1}^{\infty}a_n(F)r^n+\lambda(r)\sum_{n=2}^{\infty}(a_n(F))^2r^{2n}\leq 1\bigg\}\\&=\sup\bigg\{r:F(r)+\lambda(r)\sum_{n=2}^{\infty}(a_n(F))^2r^{2n}=1\bigg\}\\&=\sup\{r:G^\lambda_F(r)=1\}\\&=r^\lambda_F.
 \end{align*}
 This completes the proof.
 \end{proof}
 We have the following remark in connection with the proof of Lemma \ref{lem-1.1}.
 \begin{rem}
 There can be many functions $F(\not\equiv 0)$ in the class $B$ with non-negative coefficients and for each such $F$, the inequality $ \rho^\lambda(B)\leq r^\lambda_F$ holds true. However, it is worth observing that such an upper bound for the radius $\rho^\lambda(B)$ becomes more accurate when the coefficient $a_n(F)$ of the function $F$ is larger. In particular, if $\lambda\equiv 0$, then similar observations as in \cite[Remark 1]{Ponn-Shmi-Star-2024} are also true. 
 \end{rem}
 If the function $F$ realizes the maximum in this problem, then we get the following corollary.
 \begin{cor}\label{Cor-2.1}
 	If the assumptions in Lemma \ref{Cor-2.1} are satisfied, then the function $G^\lambda_F(z)$ achieves its maximum value in the given problem, hence leading to $R^\lambda_*=r^\lambda_F$ and $\rho^\lambda(B)$ as the only solution to $G^\lambda_F(r)=1$.
 \end{cor}
 \begin{rem}
  If $\lambda=0$, then we see that Corollary \ref{Cor-2.1} reduces to \cite[Corollary 1]{Ponn-Shmi-Star-2024}.
 \end{rem}
  Let us use Corollary \ref{Cor-2.1} to obtain a refined Bohr radius for  the class $\mathfrak{L}\mathcal{K}$ and the corresponding Bohr radius as root in $(0, 1)$ of an equation.
 \begin{thm}\label{Th-2.1}
The refined Bohr radius $\rho^\lambda(\mathfrak{L}\mathcal{K})\in (0, 1)$ for the class $\mathfrak{L}\mathcal{K}$ is the root of equation 
\begin{align}\label{Eq-2.2}
	\log\left(\frac{1}{(1-r)^2}\right)+4\lambda(r)({\rm Li}_2(r^2)-r^2)=1.
\end{align} 
 \end{thm}	
\begin{proof}
Let $f(z) =\sum_{n=1}^{\infty}a_nz^n\in\mathfrak{L}\mathcal{K}$. From the integral representation in the class $\mathcal{K}$, one directly obtains	the well-known (see \cite[Chapter 2]{Pommerenke-1975}) sharp estimate for the coefficients $|a_n|\leq\frac{2}{n}$ for
functions in $\mathfrak{L}\mathcal{K}$. As $f_0(z) = z/(1-z)$ belongs to $\mathcal{K}$ and $f_0^{\prime}(z) = 1/(1-z)^{2}$, it follows that the function $F(z)=\log(1-z)^{-2}=\sum_{n=1}^{\infty}\frac{2}{n}z^n\in\mathfrak{L}\mathcal{K}$ is extremal.\vspace{1.2mm}
	
Applying Corollary \ref{Cor-2.1}, $\rho^\lambda(\mathfrak{L}\mathcal{K})$ is obtained by solving the equation
	\begin{align*}
		F(r)+\lambda(r)\sum_{n=2}^{\infty}|a_n(f)|^2r^{2n}=1.
	\end{align*}
	That is $\rho^\lambda(\mathfrak{L}\mathcal{K})$ is a root in $(0, 1)$ of the equation
	\begin{align*}
		\log\left(\frac{1}{(1-r)^2}\right)+4\lambda(r)({\rm Li}_2(r^2)-r^2)=1.
	\end{align*}
	This completes the proof.
\end{proof}
\begin{table}[ht]
	\centering
	\begin{tabular}{|l|l|}
		\hline
		$\;\;\;\;\;\;\;\;\lambda(r)$& $\;\;\rho^\lambda(\mathfrak{L}\mathcal{K})$ \\
		\hline
		$\;\;\;\;\;\;\;\;\;r$& $0.390504 $\\
		\hline
		$\;\;\;\;\;\;\;\;\;r^2$& $0.39228$\\
		\hline
		$\;\;\;\;\;\;\;\;\;e^r$& $0.383116$\\
		\hline
		$\;\;\;\;\;\;\;\sin r$& $0.390576 $\\
		\hline
		$\;\;\;\;{1}/{(1-r)}$& $0.382155 $\\
		\hline
		$\;\;\;\;{r}/{(1-r)}$& $0.388724$\\
		\hline
		$\;\;\;\;{r}/{(1-r)^2}$& $0.386029$\\
		\hline
		$\;\;\;\;{r}/{(1-r)^3}$& $0.382145$\\
		\hline
		$\;\;\;{r e^r}/{(1-r)}$& $0.386682$\\
		\hline
		$\;\;\;{r e^r}/{(1-r)^2}$& $0.383059$\\
		\hline
		$\;\;\;{r e^r}/{(1-r)^3}$& $0.37808$\\
		\hline
	\end{tabular}
	\vspace{3mm}
	\caption{Values of $\rho^\lambda(\mathfrak{L}\mathcal{K})$ in $(0, 1)$ of equation \eqref{Eq-2.2} for different choice of the function $\lambda(r)$ in Theorem \ref{Th-2.1}}
	\label{tabel-2}
\end{table}
 \begin{rem}
 If $f(z) =\sum_{n=1}^{\infty}a_nz^n\in\mathfrak{L}\mathcal{K}$, then we have
 \begin{align*}
 \sum_{n=1}^{\infty}\left(\frac{2}{n}\right)r^{n}=-2\log(1-r)\; \mbox{and}\;\sum_{n=2}^{\infty}\left(\frac{2}{n}\right)^2r^{2n}=4\left(-r^2+{\rm Li}_2\left(r^2\right)\right).
 \end{align*}
 Further, for $f\in\mathfrak{L}\mathcal{K}$, in particular when $a_1=2$, if we choose the function $\lambda(r)$ by
\begin{align*}
\lambda(r)=\frac{1}{1+|a_1|}+\frac{r}{1-r}=\frac{1+2r}{3(1-r)},
\end{align*}
then \eqref{Eq-2.2} becomes 
 \begin{align*}
 -2\log(1-r)+4\left(\frac{1+2r}{3(1-r)}\right)({\rm Li}_2(r^2)-r^2)=1,
 \end{align*}
 thus we see that
 $\rho^\lambda(\mathfrak{L}\mathcal{K})\approx 0.386442$ in $(0, 1)$.
 \end{rem}
 Using Lemma \ref{lem-1.1}, we obtain the following result for the family $\mathcal{S}$.
 \begin{thm}\label{Th-2.2}
 The refined Bohr radius for the class $\mathcal{S}$ of univalent function is given by $\rho^{\lambda}(\mathcal{S})$, where $\rho^{\lambda}(\mathcal{S})$ is the root in $(0, 1)$ of equation
 \begin{align}\label{Eq-2.3}
 \frac{r}{(1-r)^2}+\lambda(r)\frac{r^4(r^4-3r^2+4)}{(1-r^2)^3}=1.
 \end{align}
 \end{thm}
 \begin{proof}
 Let $f(z) = z +\sum_{n=2}^{\infty}z^n\in \mathcal{S}$. We see that the class S satisfies the conditions of Lemma 1. Then the Bieberbach conjecture (see \cite{Bieberbach-1916}), proved by de Branges (see \cite{Branges-1985}), concerning the class S states that $|a_n|\leq n, n\in\mathbb{N}$.
 Equalities are achieved for the Koebe function $k(z) = \frac{z}{(1-z)^2}\in\mathcal{S}$ for any $n\in\mathbb{N}$. Clearly, Corollary \ref{Cor-2.1} is
 applicable. Accordingly, the Bohr radius $\rho^\lambda(\mathcal{S})$ is the root of the equation $ G^\lambda_F(r)=1 $, which is 
 \begin{align*}
 k(r)+\lambda(r)\sum_{n=2}^{\infty}n^2r^{2n}=1.
 \end{align*}
Thus, we see that $\rho^\lambda(\mathcal{S})$ is a root in $(0, 1)$ of equation
 \begin{align*}
 	\frac{r}{(1-r)^2}+\lambda(r)\frac{r^4(r^4-3r^2+4)}{(1-r^2)^3}=1.
 \end{align*}
This completes the proof.
 \end{proof}
 Since $\lambda(r)$ is a real valued function from $[0,1]$ to $[0,\infty)$,  we consider $ \lambda(r) $ in Theorem \ref{Th-2.2}  as
 \begin{align}\label{EQ-2.2}
 	\lambda(r)=\frac{1}{1+|a_1|}+\frac{r}{1-r},
 \end{align}
 and obtain the following immediate corollary.
\begin{cor}
	The refined Bohr radius for the class $\mathcal{S}$ of univalent function with the choice of $\lambda$ as \eqref{EQ-2.2}, we have $\rho^\lambda(\mathcal{S})\approx0.363379$.
\end{cor}

\begin{table}[ht]
	\centering
	\begin{tabular}{|l|l|}
		\hline
		$\;\;\;\;\;\;\;\;\;\lambda(r)$& $\;\;\;\rho^\lambda(\mathcal{S})$ \\
		\hline
		$\;\;\;\;\;\;\;\;\;\;\;r$& $0.374675 $\\
		\hline
		$\;\;\;\;\;\;\;\;\;\;\;r^2$& $0.379046$\\
		\hline
		$\;{1}/{2}+{r}/{(1-r)}$& $0.363379$\\
		\hline
		$\;\;\;\;\;\;\;\;\;\;\;e^r$& $0.358379$\\
		\hline
		$\;\;\;\;\;\;\;\;\;\sin r$& $0.37483$\\
		\hline
		$\;\;\;\;\;{r}/{(1-r)}$& $0.370916$\\
		\hline
		$\;\;\;\;\;{r}/{(1-r)^2}$& $0.365787$\\
		\hline
		$\;\;\;\;\;{r}/{(1-r)^3}$& $0.359251$\\
		\hline
		$\;\;\;\;\;{r}/{(1-r)^4}$& $0.351496$\\
		\hline
		$\;\;\;\;\;{r e^r}/{(1-r)}$& $0.366913$\\
		\hline
		$\;\;\;\;\;{r e^r}/{(1-r)^2}$& $0.370916$\\
		\hline
		$\;\;\;\;\;{r}/{(1-r)^2}$& $0.360621$\\
		\hline
		$\;\;\;\;{r e^r}/{(1-r)^3}$& $0.353043$\\
		\hline
		$\;\;\;\;{r e^r}/{(1-r)^4}$& $0.344504$\\
		\hline
	\end{tabular}
	\vspace{3mm}
	\caption{Values of $\rho^\lambda(\mathcal{S})$ in $(0, 1)$ of equation \eqref{Eq-2.3} for different choice of the function $\lambda(r)$ in Theorem \ref{Th-2.2}}
\end{table}
\section{Bohr Inequality of discrete Fourier transforms for simply connected domain $ \Omega_\gamma $}
 The Bohr inequality other than class of analytic or harmonic mappings, holomorphic mappings, also studied for certain operators. A well-known operator is Ces$\acute{a}$ro operator which is studied in \cite{Hardy-Littlewood-MZ-1932} and defined as
 \begin{align*}
 	T[f](z):=\sum_{n=0}^{\infty}\bigg(\frac{1}{n+1}\sum_{k=0}^{n}a_k\bigg)z^n=\int_{0}^{1}\frac{f(tz)}{1-tz}dt,
 \end{align*}
 where $f(z)=\sum_{n=0}^{\infty}a_nz^n$ is analytic in $\mathbb{D}$. The Bohr radius for Ces$\acute{a}$ro operator and integral operators are obtained in recent study. For example,  Kayumov \emph{et al.} (see \cite{Kayumov-MJM-2022}) generalized the classical Bohr theorem introducing a sequence of non-negative function $\{\varphi_k(r)\}_{k=0}^\infty$ and as an application, also they obtained a counterpart of Bohr theorem for the generalized $\alpha$-Ces$\acute{a}$ro operator $\mathcal{C}_f^\alpha(r)$, where $\alpha\in\mathbb{C}$, with ${\rm Re}\;\alpha>-1$. Kumar and Sahoo (see \cite{Kumar-Sahoo-MJM-2021}) have obtained the Bohr radii for certain complex integral operators and Bernardi operator (see also \cite[p. 11]{Miller-2000})
 \begin{align*}
 	L\beta[f](z):=\sum_{n=m}^{\infty}\frac{a_n}{n+\beta}z^n=\int_{0}^{1}f(zt)t^{\beta-1}dt,\;\mbox{where}\;\beta>-m,
 \end{align*}
 defined on a set of bounded analytic functions in the unit disk. Later, Kumar (see \cite{Kumar-CVEE-Integral-2023}) have generalized Bohr inequality established by Kayumov \emph{et al.} (see \cite{Kayumov-MJM-2022}) for unit disk $\mathbb{D}$ to the class of bounded analytic functions defined on the simply connected domain $\Omega_{\gamma}$. Allu and Ghosh (see \cite{Allu-Ghosh-2023}) studied the Bohr type inequality for Ces$\acute{a}$ro operator established by (see \cite{Kayumov-CRA-2020})  and Bernardi integral operator established by Kumar and Sahoo (see \cite{Kumar-CVEE-Integral-2023}) to  the space of analytic functions defined on a simply connected domain $\Omega_{\gamma}$ containing the unit disk $\mathbb{D}$. \vspace{1.2mm}

 Let $\{y_n\}_{n=0}^{N-1}$ be a sequence of complex numbers. The discrete Fourier transform is defined by 
\begin{align*}
	x_k=\sum_{n=0}^{N-1}y_ne^{-\frac{2\pi ink}{N}},
\end{align*}
whereby $\mathcal{F}(y_n)=(x_k)$ (see \cite{Bachman-2000}). For $f(z)=\sum_{n=0}^{\infty}a_nz^n\in\mathcal{B}$, we perform the discrete Fourier transform on the coefficients $a_k$ from $k=0$ to $n$ which gives 
\begin{align*}
	\mathcal{F}[f](z)=\sum_{n=0}^{\infty}\bigg(\sum_{k=0}^{n}a_ke^{-\frac{2\pi ink}{n+1}}\bigg)z^n.
\end{align*}
To obtain a Bohr-type inequality for simply connected domain $ \Omega_{\gamma} $, we denote the majorant series of $ \mathcal{F}[f](z) $ as
\begin{align*}
	\mathcal{F}_f(r):=\sum_{n=0}^{\infty}\bigg(\sum_{k=0}^{n}\big|a_ke^{-\frac{2\pi ink}{n+1}}\big|\bigg)r^n,
\end{align*}
where $r=|z|<1$.\vspace{1.2mm} 

Actually, for functions defined on the unit disk $ \Omega_0=\mathbb{D} $, Ong \emph{et al.} (see \cite{Ong-2024}) studied the Bohr inequality for Fourier integral transform and Laplace transform and obtained the following results.
\begin{thmC}\cite[Theorem 1]{Ong-2024}
Let $f(z)=\sum_{n=0}^{\infty}a_nz^n\in\mathcal{B}$. Then,
\begin{align*}
	\mathcal{F}_f(r)\leq\frac{1}{1-r}
\end{align*}
whenever $r\leq 1/3$. The constant $1/3$ cannot be improved.
\end{thmC}
In \cite{Ong-2024}, the discrete Laplace transform, $ \mathcal{L} $, is considered, such that for a sequence $ \{y_n\}_{n=0}^{\infty} $, $ \mathcal{L}(y_n)=(x_k) $, where 
\begin{align*}
	(x_k)=\sum_{n=0}^{\infty}\frac{y_n}{(k+1)^{n+1}},
\end{align*}
when the series on the right hand side is convergent. Thus it follows that
\begin{align*}
	\mathcal{L}_f(r)=\sum_{n=0}^{\infty}\left(\sum_{k=0}^{n}\frac{|a_k|}{(n+1)^{k+1}}\right)r^n,
\end{align*}
where $ |z|=r<1 $.
\begin{theoD}\cite[Theorem 3]{Ong-2024}
	If  $f(z)=\sum_{n=0}^{\infty}a_nz^n\in\mathcal{B}$, then 
	\begin{align*}
		\mathcal{L}_f(r)\leq\frac{1}{r}\ln\bigg(\frac{1}{1-r}\bigg)
	\end{align*}
	for all $0<r<1$.
\end{theoD}
Through a detailed study, we see that Theorem C and Theorem D can be applied to functions with domains that extend beyond the unit disk. In fact, the aforementioned discussions have inspired us to pose the following questions for further study.
\begin{ques}\label{Q-2}
	Can we establish the Theorem C and Theorem D for functions in simply connected domain $\Omega_\gamma$?
\end{ques}
\begin{ques}\label{Q-3}
	Can we obtain a better estimate in Theorem D? Can we show the inequality in Theorem D is sharp?
\end{ques}
We affirmatively answer the above questions and obtain Theorem \ref{TH-1.1} and Theorem \ref{Th-2.2}. We show that, if we change the method of proof, then we can obtain a better estimate as well as the sharp inequality in Theorem D. In order to establish our theorems, we require the following Lemma, which was established by Evdoridis \textit{et al.} (see \cite{Evd-Ponn-Rasi-2020}), regarding the coefficient bound of functions in a simply connected domain $\Omega_\gamma$. This lemma will play a crucial role in proving our theorems.
\begin{lemA}\cite[Lemma 2]{Evd-Ponn-Rasi-2020}
For $\gamma\in[0,1)$, let $f$ be an analytic function in $\Omega_\gamma$ bounded by $1$ with the series representation $f(z)=\sum_{n=0}^{\infty}a_nz^n$ in the unit disk $\mathbb{D}$. Then
	\begin{align*}
		|a_n|\leq\frac{1-|a_0|^2}{1+\gamma}\;\;\mbox{for}\;n\geq1.
	\end{align*}
\end{lemA}
We are in a state to give our first theorem as an answer to the Question \ref{Q-2}.
\begin{thm}\label{TH-1.1}
	For $0\leq\gamma<1$, let $f\in\mathcal{B}(\Omega_\gamma)$ with $f(z)=\sum_{n=0}^{\infty}a_nz^n$ in $\mathbb{D}$. Then 
	\begin{align*}
		\mathcal{F}_f(r)\leq\frac{1}{1-r}
	\end{align*}
	for $r\leq r_0=(1+\gamma)/{(3+\gamma)}$. The radius $r_0$ is best possible.
\end{thm}
\begin{rem}
	By putting $\gamma=0$, we obtain exactly \cite[Theorem 1]{Ong-2024} for unit disk $\mathbb{D}$.
\end{rem}
\begin{proof}[\bf Proof of Theorem \ref{TH-1.1}]
	The function $f$ be analytic in $\Omega_\gamma$, with $|f(z)|<1$. First, we note that
	\begin{align*}
		\mathcal{F}_f(r)=\sum_{n=0}^{\infty}\bigg(\sum_{k=0}^{n}|a_k|\bigg)r^n.
	\end{align*}
	By Lemma A, we have $|a_k|\leq(1-|a_0|^2)/(1+\gamma)$. A routine computation yields that
	\begin{align*}
	\mathcal{F}_f(r)\leq|a_0|\sum_{n=0}^{\infty}r^n+\frac{1-|a_0|^2}{1+\gamma}\sum_{n=0}^{\infty}nr^n=\frac{|a_0|}{1-r}+\frac{(1-|a_0|^2r)}{(1+\gamma)(1-r)^2}.
	\end{align*}
	For $|a_0|=x\in (0, 1)$, we define the function
	\begin{align*}
	\Psi_{\gamma,r}(x):=\frac{x}{1-r}+\frac{(1-x^2)r}{(1+\gamma)(1-r)^2}\; \mbox{for}\; r,\gamma\in[0,1).
	\end{align*}
	Differentiating $\Psi_{\gamma,r}(x)$ w.r.t. $x$, we see that
	\begin{align*}
		(\Psi_{\gamma,r})^{\prime}(x)=\frac{1}{1-r}-\frac{2xr}{(1+\gamma)(1-r)^2}
	\end{align*}
	and
	\begin{align*}
		(\Psi_{\gamma,r})^{\prime\prime}(x)=-\frac{2r}{(1+\gamma)(1-r)^2}\leq0\;\;\mbox{for}\;r,\gamma\in[0,1).
	\end{align*}
	Evidently, $(\Psi_{\gamma,r})^{\prime}(x)$ is a decreasing function of $x$. Hence, 
	\begin{align*}
		(\Psi_{\gamma,r})^{\prime}(x)\geq(\Psi_{\gamma,r})^{\prime}(1)=\frac{(1-r)(1+\gamma)-2r}{(1+\gamma)(1-r)^2}.
	\end{align*}
	Then, it is clear that $(\Psi_{\gamma,r})^{\prime}(1)\geq 0$ for $r\leq r_0=(1+\gamma)/(3+\gamma)$. Thus we see that  $(\Psi_{\gamma,r})^{\prime}(x)\geq 0$ when $r\leq r_0$. For $r\leq r_0, \Psi_{\gamma,r}(x)$ being an increasing function yields that
	\begin{align*}
		\Psi_{\gamma,r}(x)\leq \Psi_{\gamma,r}(1)=\frac{1}{1-r}.
	\end{align*}
	The inequality we desire has now been established.\vspace{1mm}
	
\noindent To show that the radius $r_0$ is best possible, we consider the function 
	\begin{align}\label{EQQ-3.1}
		f_0(z)=\frac{a-\gamma-(1-\gamma)z}{1-a\gamma-a(1-\gamma)z}=A_0-\sum_{n=1}^{\infty}A_nz^n,\;z\in\mathbb{D},
	\end{align} 
	where
	\begin{align*}
		\begin{cases}
			A_0=\dfrac{a-\gamma}{1-a\gamma},\vspace{2mm}\\A_n=\dfrac{1-a^2}{a(1-a\gamma)}\bigg(\dfrac{a(1-\gamma)}{1-a\gamma}\bigg)^n\; \mbox{for}\; n\in\mathbb{N}.
		\end{cases}
	\end{align*}
	With this function $ f_0 $, a routine computation shows that
	\begin{align*}
		\mathcal{F}_{f_0}(r)&=\frac{a-\gamma}{1-a\gamma}\frac{1}{1-r}+\frac{1-a^2}{a(1-a\gamma)}\sum_{n=1}^{\infty}\bigg(\sum_{k=1}^{n}\bigg(\frac{a(1-\gamma)}{1-a\gamma}\bigg)^k\bigg)r^n\\&=\frac{a-\gamma}{1-a\gamma}\frac{1}{1-r}+\frac{1-a^2}{a(1-a\gamma)}\sum_{n=1}^{\infty}\frac{a(1-\gamma)}{1-a}\bigg(1-\bigg(\frac{a(1-\gamma)}{1-a\gamma}\bigg)^n\bigg)r^n\\&=\frac{a-\gamma}{1-a\gamma}\frac{1}{1-r}+\frac{(1+a)(1-\gamma)}{(1-a\gamma)}\sum_{n=1}^{\infty}\bigg(1-\bigg(\frac{a(1-\gamma)}{1-a\gamma}\bigg)^n\bigg)r^n\\&=\frac{1}{1-r}-\frac{1}{1-r}\Phi^\gamma_a(r),
		\end{align*} 
		where
		\begin{align*}
			\Phi^\gamma_a(r):=1-\frac{a-\gamma}{1-a\gamma}+\frac{(1+a)(1-\gamma)}{(1-a\gamma)}\bigg(-r+\frac{a(1-\gamma)r}{1-a\gamma-ar+a\gamma r}\bigg).
		\end{align*}
		A further computation shows that
		\begin{align*}
			(\Phi^\gamma_a)^{\prime}(r)=-\frac{(1-\gamma)(1-a^2)}{(1-a\gamma(1-r)-ar)^2}<0\;\;\mbox{for}\;r\in(0,1),
		\end{align*}
	hence, $(\Phi^\gamma_a)(r)$ is an increasing function of $r$ in $(0,1)$ and thus for $r>r_0=(1+\gamma)/(3+\gamma)$, we have
		\begin{align*}
				(\Phi^\gamma_a)(r)<(\Phi^\gamma_a)(r_0)=\frac{2(1-a)^2(1+\gamma)^2}{(1-a\gamma)(3-a+\gamma-3a\gamma)},
		\end{align*}
		which tends to $0$ as $a\to1^-$. Therefore, $(\Phi^\gamma_a)(r)$ is negative for $r>r_0$, and hence 
		\begin{align*}
			\frac{1}{1-r}-\frac{1}{1-r}\Phi^\gamma_a(r)>\frac{1}{1+r}.
		\end{align*}
		This completes the proof.
\end{proof}
\begin{thm}\label{TH-1.2}
		For $0\leq\gamma<1$, let $f\in\mathcal{B}(\Omega_\gamma)$ with $f(z)=\sum_{n=0}^{\infty}a_nz^n$ in $\mathbb{D}$. Then 
	\begin{align}\label{Eq-3.1}
		\mathcal{L}_f(r)\leq\frac{1}{r}\ln\bigg(\frac{1}{1-r}\bigg)
	\end{align}
	for $0<r<1$.
\end{thm}
\begin{rem}
In particular, when $\gamma=0$, we get exactly the result \cite[Theorem 3]{Ong-2024} for unit disk $\mathbb{D}$. In fact, it is worth pointing out that the upper bound in the inequality \eqref{Eq-3.1} is independent of $\gamma\in [0, 1)$ in Theorem \ref{TH-1.2}. Thus, Theorem \ref{TH-1.2} is a more compact form of \cite[Theorem 3]{Ong-2024}. 
\end{rem}
\begin{proof}[\bf Proof of Theorem \ref{TH-1.2}]
The function $f$ be analytic in $\Omega_\gamma$, with $|f(z)|<1$. First, we note that
\begin{align*}
\mathcal{L}_f(r)=\sum_{n=0}^{\infty}\bigg(\sum_{k=0}^{n}\frac{|a_k|}{(n+1)^{k+1}}\bigg)r^n.
\end{align*}
By Lemma A, a tedious computation, using the estimate
\begin{align*}
-\sum_{n=1}^{\infty}\frac{r^n}{n(n+1)^{n+1}}<0,
\end{align*}
we see that
\begin{align*}
\mathcal{L}_f(r)&\leq |a_0|\sum_{n=0}^{\infty}\frac{r^n}{n+1}+\frac{1-|a_0|^2}{1+\gamma}\sum_{n=1}^{\infty}\bigg(\sum_{k=1}^{n}\frac{1}{(n+1)^{k+1}}\bigg)r^n\\&=-\frac{|a_0|}{r}\ln(1-r)+\frac{1-|a_0|^2}{1+\gamma}\sum_{n=1}^{\infty}\bigg(\frac{1}{n(n+1)}-\frac{1}{n(n+1)^{n+1}}\bigg)r^n\\&<-\frac{|a_0|}{r}\ln(1-r)+\frac{1-|a_0|^2}{1+\gamma}\bigg[\bigg(\frac{1}{r}-1\bigg)\ln(1-r)+1\bigg]\\&=\frac{\{-|a_0|(1+\gamma)+(1-|a_0|^2)(1-r)\}\ln(1-r)}{r(1+\gamma)}+\frac{1-|a_0|^2}{1+\gamma}\\&:=\Psi_1^\gamma(|a_0|)
	\end{align*}
  Let $|a_0|=x\in (0, 1)$. We see that
	\begin{align*}
		(\Psi_1^\gamma)^\prime(x)=\frac{\{2xr-2x-(1+\gamma)\}\ln(1-r)-2xr}{r(1+\gamma)}
	\end{align*}
	and 
	\begin{align*}
		(\Psi_1^\gamma)^{\prime\prime}(x)=-\frac{2r(1-\ln(1-r))+2\ln(1-r)}{(1+\gamma)r}\leq 0 \;\;\mbox{for all}\;\;x,r,\gamma\in[0,1).
	\end{align*}
	Evidently, $(\Psi_1^\gamma)^\prime(x)$ is a decreasing function of $x$, which implies that 
	\begin{align*}
	(\Psi_1^\gamma)^\prime(x)\geq(\Psi_1^\gamma)^\prime(1)=\frac{\{2r-2-(1+\gamma)\}\ln(1-r)-2r}{r(1+\gamma)}=\frac{\Phi^\gamma(r)}{r(1+\gamma)},
	\end{align*}
	where
	\begin{align*}
	\Phi^\gamma(r)=\{2r-2-(1+\gamma)\}\ln(1-r)-2r.
	\end{align*}
 We observe that $\Phi^\gamma(r)$ is an increasing function because of the fact that
 \begin{align*}
 (\Phi^\gamma)^{\prime}(r)=\frac{1+\gamma}{1-r}+2\ln(1-r)>0\;\;\mbox{for all}\;\;r\in(0,1).
 \end{align*}
 Consequently, we have $\Phi^\gamma(r)>\Phi^\gamma(0)=0$. Hence, $(\Psi_1^\gamma)^\prime(x)>0$ for all $x\in (0, 1)$. Finally, we have
 \begin{align*}
 \Psi_1^\gamma(|a_0|)\leq\Psi_1^\gamma(1)=\frac{1}{r}\ln\bigg(\frac{1}{1-r}\bigg).
 \end{align*}
 This completes the proof.
\end{proof}
\subsection{Concluding remark on Theorem \ref{TH-1.2}}
We have used the estimate
\begin{align}\label{Eq-3.2}
	-\sum_{n=1}^{\infty}\frac{r^n}{n(n+1)^{n+1}}<0,
\end{align} in the proof of Theorem \ref{TH-1.2}. However, the proof can be modified without using \eqref{Eq-3.2}. For example, using the estimate
\begin{align*}
	\sum_{k=1}^{n}\frac{1}{(n+1)^{k+1}}\leq \sum_{k=1}^{n}\frac{1}{(n+1)^2}=\frac{n}{(n+1)^2}
\end{align*}
we see that
\begin{align*}
	\sum_{n=0}^{\infty}\left(\sum_{k=1}^{n}\frac{1}{(n+1)^{k+1}}\right)r^n\leq \sum_{n=0}^{\infty}\left(\sum_{k=1}^{n}\frac{n}{(n+1)^{2}}\right)r^n=\frac{1}{r}\ln\left(\frac{1}{1-r}\right)-\frac{{\rm Li}_2(r)}{r}.
\end{align*}
Thus, we have
\begin{align}\label{Eq-3.3}
	\mathcal{L}_f(r)\leq \frac{a}{r}\ln\left(\frac{1}{1-r}\right)+\frac{1-a^2}{1+\gamma}\left(\frac{1}{r}\ln\left(\frac{1}{1-r}\right)-\frac{{\rm Li}_2(r)}{r}\right):=F_{\gamma,r}(a).
\end{align}
A simple computation gives us
\begin{align*}
	\frac{d}{da}(F_{\gamma,r}(a))=\frac{1}{r}\ln\left(\frac{1}{1-r}\right)-\frac{2a}{1+\gamma}\left(\frac{1}{r}\ln\left(\frac{1}{1-r}\right)-\frac{{\rm Li}_2(r)}{r}\right)
\end{align*}
 and 
 \begin{align*}
 	\frac{d^2}{da^2}(F_{\gamma,r}(a))=-\frac{2}{1+\gamma}\left(\frac{1}{r}\ln\left(\frac{1}{1-r}\right)-\frac{{\rm Li}_2(r)}{r}\right)\leq 0
 \end{align*}
 for all $r,\gamma\in (0, 1)$.\vspace{2mm}
 
 \noindent Clearly, $\frac{d}{da}(F_{\gamma,r}(a))$ is a decreasing function of $a\in (0, 1)$, hence 
 \begin{align*}
 	\frac{d}{da}(F_{\gamma,r}(a))\geq \frac{d}{da}(F_{\gamma,r}(1))=\frac{1}{r}\ln\left(\frac{1}{1-r}\right)-\frac{2}{1+\gamma}\left(\frac{1}{r}\ln\left(\frac{1}{1-r}\right)-\frac{{\rm Li}_2(r)}{r}\right)\geq 0
 \end{align*}
 for all $a\in (0, 1)$. This shows that $F_{\gamma,r}(a)$ is an increasing function of $a\in (0, 1)$. Consequently, we have 
\begin{align*}
\mathcal{L}_f(r)\leq F_{\gamma,r}(a)\leq F_{\gamma,r}(1)=\frac{1}{r}\ln\left(\frac{1}{1-r}\right)\; \mbox{for all}\; r\in (0, 1).
\end{align*}
The same estimate, as in \cite{Ong-2024}, is obtained. Our observation is that regardless of how we arrange the series, the upper bound of $\mathcal{L}_f(r)$ will remain the same due to the presence of the factor $(1-a^2)$ in the expression of the function $F_{\gamma,r}(a)$, as we take $a\to 1^{-}$.\vspace{2mm}

However, \cite{Ong-2024} does not accurately determine the Bohr radius. In contrast, our study shows that it is possible to find it by rearranging the terms using a different approach. Let us rewrite the right side of \eqref{Eq-3.3} in the following manner:
\begin{align*}
\mathcal{L}_f(r)&\leq \frac{1}{r}\ln\left(\frac{1}{1-r}\right)+\frac{(1-a)}{r}\ln (1-r)+ \frac{2(1-a)}{1+\gamma}\left(\frac{1}{r}\ln\left(\frac{1}{1-r}\right)-\frac{{\rm Li}_2(r)}{r}\right)\\&\leq\frac{1}{r}\ln\left(\frac{1}{1-r}\right)+\frac{(1-a)}{(1+\gamma)}\bigg[\frac{(1+\gamma)}{r}\ln(1-r)+2\left(\frac{1}{r}\ln\left(\frac{1}{1-r}\right)-\frac{{\rm Li}_2(r)}{r}\right)\bigg]\\&\leq \frac{1}{r}\ln\left(\frac{1}{1-r}\right)+\frac{(1-a)}{(1+\gamma)}\Phi_{\gamma}(r),
\end{align*}
where 
\begin{align*}
	\Phi_{\gamma}(r):=\frac{(1+\gamma)}{r}\ln(1-r)+2\left(\frac{1}{r}\ln\left(\frac{1}{1-r}\right)-\frac{{\rm Li}_2(r)}{r}\right).
\end{align*}
Thus, the desired inequality
\begin{align}\label{Eq-3.5}
	\mathcal{L}_f(r)\leq \frac{1}{r}\ln\left(\frac{1}{1-r}\right)
\end{align}
can be obtained if $\Phi_{\gamma}(r)\leq 0$ for $r\leq r_{\gamma}<1$.  By using basic theorems in calculus, one can obtain $r_{\gamma}$ as a root in $(0, 1)$ of the equation $\Phi_{\gamma}(r)=0$ for values of $\gamma\in [0, \gamma_*]$, where $\gamma_*\approx 0.27713$. Consequently, we see that $r_{\gamma}$ is the Bohr radius for the discrete Laplace transform $ \mathcal{L} $. It is important to note that, specifically, when $\Omega_0=\mathbb{D}$, we observe that the root of the following equation is $r_0\approx 0.940599$,
\begin{align*}
	\frac{1}{r}\ln(1-r)+\frac{2}{r}\left(\ln\left(\frac{1}{1-r}\right)-{\rm Li}_2(r)\right)=0.
\end{align*}
Hence, we see that a better estimate is obtained in the radius.\vspace{1.2mm}

Moreover, it can be shown that the inequality \eqref{Eq-3.5} is sharp and this can be shown using the function $f_0$ as given by \eqref{EQQ-3.1}. When $f$ is equal to $f_0$, we observe that
\begin{align*}
	\mathcal{L}_{f_0}(r)&=|A_0|\sum_{n=0}^{\infty}\frac{r^n}{n+1}+\sum_{n=1}^{\infty}\left(\sum_{k=1}^{n}\frac{|A_n|}{(n+1)^{k+1}}\right)r^n\\&=\left(\frac{a-\gamma}{1-a\gamma}\right)\frac{1}{r}\ln\left(\frac{1}{1-r}\right)+\frac{1-a^2}{a(1-a\gamma)}\sum_{n=1}^{\infty}\left(\sum_{k=1}^{n}\frac{1}{(n+1)^{k+1}}\left(\frac{a(1-\gamma)}{1-a\gamma}\right)^k\right)r^k\\&=\frac{1}{r}\ln\left(\frac{1}{1-r}\right)+\frac{1-a}{1-a\gamma}\bigg[-\frac{(1+\gamma)}{r}\ln\left(\frac{1}{1-r}\right)\\&\quad+\frac{1+a}{a}\sum_{n=1}^{\infty}\left(\sum_{k=1}^{n}\frac{1}{(n+1)^{k+1}}\left(\frac{a(1-\gamma)}{1-a\gamma}\right)^k\right)r^k\bigg].
\end{align*}
Taking $a\to 1^{-}$ in the last expression, it is easy to see that 
\begin{align*}
	\mathcal{L}_f(r)=\frac{1}{r}\ln\left(\frac{1}{1-r}\right)
\end{align*}
which shows that the bound is sharp. In response to \cite[Remark 1]{Ong-2024}, we now say that our result improves the estimate in the radius. Additionally, we show that a sharper inequality can be established in certain situations. \vspace{2mm}

\noindent{\bf Acknowledgment:} The first author is supported by SERB File No. SUR/2022/002244, Govt. of India and the second author is supported by UGC-JRF, File No. (NTA Ref. No. 211610135410), Govt. of India, and the third author is supported by UGC-JRF (NTA Ref. No. 201610135853), New Delhi, Govt. of India.

\vspace{2mm}

\noindent\textbf{Compliance of Ethical Standards:}\\

\noindent\textbf{Conflict of interest.} The authors declare that they have no conflicts of interest regarding the publication of this paper.\vspace{1.5mm}

\noindent\textbf{Funds.} No funds.\vspace{1.5mm}

\noindent\textbf{Data availability statement.}  Data sharing not applicable to this article as no datasets were generated or analysed during the current study.

\end{document}